\documentclass[a4paper,11pt, reqno]{amsart}

\usepackage{amsthm}
\usepackage{amssymb}
\usepackage{amsmath}
\usepackage{mathtools}
\usepackage{pdfpages}
\usepackage{changes}
\usepackage{float}
\usepackage{tikz-cd}
\usepackage{exscale,cite,color,amsopn}
\usepackage[colorlinks=true,pdftex,unicode=true,linktocpage,bookmarksopen,hypertexnames=false]{hyperref}
\usepackage{colortbl}

\renewcommand{\leq}{\leqslant}

\DeclareMathOperator{\id}{id}

\DeclareMathOperator{\Ret}{Ret}

\DeclareMathOperator{\Sym}{Sym}

\newcommand{\N}{\mathbb{N}}
\newcommand{\Z}{\mathbb{Z}}

\newcommand{\D}{\mathbb{D}}
\newcommand{\GAP}{\textsf{GAP}}

\makeatletter
\numberwithin{equation}{section}
\numberwithin{figure}{section}
\numberwithin{table}{section}
\newtheorem{thm}{Theorem}[section]
\newtheorem*{thm*}{Theorem}
\newtheorem{lem}[thm]{Lemma}

\theoremstyle{definition} 

\newtheorem{convention}[thm]{Convention}
\newtheorem{rem}[thm]{Remark}
\newtheorem{exa}[thm]{Example}

\makeatother

\title[Enumeration of solutions]{Enumeration of set-theoretic solutions to the Yang--Baxter equation}
\author{\"O. Akg\"un, M. Mereb, L. Vendramin}

\address[M. Mereb and L. Vendramin]{IMAS--CONICET and Depto. de Matem\'atica, FCEN, Universidad de Buenos Aires, Pab.~1,
Ciudad Universitaria, C1428EGA, Buenos Aires, Argentina}
\email{lvendramin@dm.uba.ar}
\email{mmereb@dm.uba.ar}

\address[L. Vendramin]{Department of Mathematics, Vrije Universiteit Brussel, Pleinlaan 2, 1050 Brussel, Belgium}
\email{Leandro.Vendramin@vub.be}

\address{School of Computer Science, University of St Andrews, St Andrews, Fife KY16 9SX, UK.}
\email{ozgur.akgun@st-andrews.ac.uk}


\begin{document}

\begin{abstract}
	We use Constraint Satisfaction methods to enumerate and construct
	set-theoretic solutions to the Yang--Baxter equation of small size. We show that there
	are 321,931 involutive solutions of size nine, 4,895,272 involutive solutions 
	of size ten and 422,449,480 non-involutive solution of size eight. 
	Our method is then used to enumerate non-involutive biquandles. 
\end{abstract}
 
\maketitle

\section{Introduction}

The Yang--Baxter equation (YBE) was 
first introduced in the field of statistical mechanics and for several decades it has been studied 
in mathematics and physics \cite{MR0290733,MR0261870}. Recent progress in set-theoretic solutions to the YBE 
shed new light on the importance of this 
equation in algebra and combinatorics \cite{MR3861714,MR2927367,MR1637256,MR1722951,MR3558231,MR1769723,MR3291816,MR2024436}.

A \emph{set-theoretic solution} to the YBE is a pair $(X,r)$, where
$X$ is a set and $r\colon X\times X\to X\times X$ 
is a bijective map such that 
\[
(r\times\id)(\id\times r)(r\times\id)=
(\id\times r)(r\times\id)(\id\times r)
\]
holds, where the juxtaposition denotes the usual composition of maps. Note that this is a functional equation in 
the space of maps $X^3\to X^3$, where $X^3=X\times X\times X$. 

\begin{exa}[permutation solutions over sets]
\label{exa:permutation}
If $X$ 
is a non-empty set and $\sigma\colon X\to X$ and $\tau\colon X\to X$ are 
bijections, 
then the pair $(X,r)$, where 
\[
r\colon X\times X\to X\times X,
\quad 
r(x,y)=(\sigma(y),\tau(x)),
\]
is a set-theoretical solution if and only if $\sigma$ and $\tau$ commute. Indeed, on the one hand, 
\begin{align*}
    (r\times\id)(\id\times r)(r\times\id)(x,y,z) &= 
    (r\times\id)(\id\times r)(\sigma(y),\tau(x),z)\\
    &=(r\times\id)(\sigma(y),\sigma(z),\tau^2(x))\\
    &=(\sigma^2(z),\tau\sigma(y),\tau^2(x))
\shortintertext{and, on the other hand,}
    (\id\times r)(r\times\id)(\id\times r)(x,y,z) &= 
    (\id\times r)(r\times\id)(x,\sigma(z),\tau(y))\\
    &=(\id\times r)(\sigma^2(z),\tau(x),\tau(y))\\
    &=(\sigma^2(z),\sigma\tau(y),\tau^2(x)).
\end{align*}
\end{exa}

We say that the solutions
$(X,r)$ and $(Y,s)$ are \emph{isomorphic} if there is a bijective map $f\colon X\to Y$ such that
\[
(f\times f)r=s(f\times f).
\]

From the combinatorial perspective certain types of solutions are particularly important; these are called non-degenerate solutions. 
By convention, if $(X,r)$ is a set-theoretic solution to the YBE, we write 
\[
	r(x,y)=(\sigma_x(y),\tau_y(x)).
\]
The solution $(X,r)$ is then said to be \emph{non-degenerate} if the maps $\sigma_x$ and
$\tau_x$ are bijective for all $x\in X$. 

\begin{convention}
	A \emph{solution} will always be a non-degenerate set-theoretic solution to
	the YBE. We will only consider finite solutions.
\end{convention}

Set-theoretic solutions to the YBE attracted a lot of attention and lead to
several interesting connections between group theory, ring theory and
combinatorics. The combinatorial version of the celebrated Yang--Baxter
equation was first formulated by Drinfel'd in~\cite{MR1183474} and considered 
later in~\cite{MR1722951,MR1637256} for involutive solutions and
in~\cite{MR1769723,MR1809284} for arbitrary solutions.  Set-theoretic solutions
are known to have deep connections with bijective $1$-cocycles, ordered 
groups, groups of I-type, regular subgroups, radical rings, skew braces, nil
rings, homology theory, Hopf--Galois extensions~\cite{MR3982254,MR3763907,MR4017867,MR3291816}.

The main result in this article is an explicit classification of solutions to
the YBE of small size.  This is achieved by using some combinatorial ideas
closely connected to the Yang--Baxter equation, Constraint Satisfaction methods~\cite{10.5555/1567016.1567043,10.1007/s10601-008-9047-y} and, in particular, the constraint modelling assistant Savile Row~\cite{nightingale2017automatically},
and the computational algebra package~\GAP~\cite{GAP4}.
Similar techniques have been used to enumerate semi-groups of order $\leq 10$, 
see~\cite{distler_semigroups_2012}.

The combination of these techniques allows us to build a huge database of involutive
and non-involutive solutions to the YBE, a good and useful source of examples
that gives an explicit and direct way to approach some open problems concerning the YBE. The database is freely 
available as a library for~\GAP~
at \url{https://github.com/vendramin/enumeration}, with DOI name \verb+10.5281/zenodo.5180745+.

In \cite{MR1722951} Etingof, Schedler and Soloviev constructed all involutive solutions of size $\leq8$. 
We summarize our findings on involutive solutions in the following statement. 

\begin{thm}\
\begin{enumerate}
    \item Up to isomorphism, there are 321,931 non-degenerate involutive set-theoretic
	solutions to the Yang--Baxter equation of size nine.
	    \item Up to isomorphism, there are 4,895,272 non-degenerate involutive set-theoretic
	solutions to the Yang--Baxter equation of size ten.
\end{enumerate}
\end{thm}

Our methods can be easily adapted to construct {racks} of small size. 
Racks are particular types of solutions to the YBE that play a fundamental 
role in combinatorial knot theory, see Section \ref{racks}.   
Using the 16,023 isomorphism classes of racks of size eight,  
we obtain the following result for non-involutive solutions of size eight.

\begin{thm}
	There are 422,449,480 non-isomorphic non-degenerate non-involutive set-theoretic
	solutions to the Yang--Baxter equation of size eight. 
\end{thm}

{No previous constructions of non-involutive solutions were known.}

Our methods could be used to construct  solutions of other sizes. However, the number of non-involutive solutions of size nine 
is expected to be extremely big. With 16,023 racks of size eight we constructed 422,449,480  non-involutive solutions, so the number of non-involutive solutions of size nine is expected to be enormous, as there are 159,526 racks of size nine.

\medskip
The paper is organized as follows. In Section~\ref{IYB} we compute the number
of involutive solutions. This is done by using a constraint satisfaction
program and the language of cycle sets. The algorithm is described at the beginning of the section.  
As an application we enumerate several
types of solutions such as indecomposable, irretractable and multipermutation
solutions. We also enumerate counterexamples to a well-known conjecture of
Gateva--Ivanova~\cite{MR2095675}. Finally, in Section~\ref{YB} we use a similar algorithm and the same computational
techniques to enumerate racks, non-involutive solutions and, in particular, non-involutive biquandles. 

\section{Involutive solutions}
\label{IYB}

A solution $(X,r)$ is said to be \emph{involutive} if $r^2=\id$.  An involutive
solution $(X,r)$ is said to be \emph{irretractable} if $\tau_x\ne\tau_y$ for all $x\ne
y$. Note that this is equivalent to $\sigma_x\ne\sigma_y$ for all $x\ne y$, as
$T\sigma_xT^{-1}=\tau_x^{-1}$ holds for all $x\in X$, where $T\colon X\to X$, $T(x)=\tau^{-1}_x(x)$, see \cite[Proposition
2.2]{MR1722951}. An involutive solution $(X,r)$ 
is said to be \emph{square-free} if $T(x)=x$ for all $x\in X$, or equivalently if $r(x,x)=(x,x)$ for all $x\in X$.


If $(X,r)$ is an involutive solution, we consider over $X$ the
equivalence relation given by
\[
	x\sim y\Longleftrightarrow \tau_x=\tau_y.
\]
This equivalence relation induces an involutive solution over the set of
equivalence classes $X/{\sim}$, known as the retraction $\Ret(X,r)$ of $(X,r)$.
An involutive solution $(X,r)$ is a \emph{multipermutation} solution if there exists
$n$ such that $|\Ret^n(X,r)|=1$, where $\Ret^{n+1}(X,r)=\Ret(\Ret^n(X,r))$. 
{Multipermutation solutions 
generalize those in Example \ref{exa:permutation}.} 

The permutation group of an involutive solution $(X,r)$ is defined as the
subgroup $\mathcal{G}(X,r)$ of $\Sym_X$ generated by the set $\{\tau_x:x\in
X\}$. An involutive solution $(X,r)$ is said to be \emph{indecomposable} if the group
$\mathcal{G}(X,r)$ acts transitively on $X$ and decomposable otherwise. 

To construct all isomorphism classes of non-degenerate involutive solutions, we will use the language of cycle sets, introduced by Rump in~\cite{MR2132760}.  
A \emph{cycle set} is a pair $(X,\cdot)$, where $X$ is a set and $X\times X\to X$,
$(x,y)\mapsto x\cdot y$, is a binary operation such that the following
conditions are satisfied:
\begin{enumerate}
    \item Each map $\varphi_x\colon X\to X$, $y\mapsto x\cdot y$, is bijective, and 
    \item $(x\cdot y)\cdot (x\cdot z)=(y\cdot x)\cdot (y\cdot z)$ for all $x,y,z\in X$.
\end{enumerate}

A cycle set $(X,\cdot)$ is said to be \emph{non-degenerate} if the map $X\to X$,
$x\mapsto x\cdot x$, is bijective.  Rump proved that non-degenerate involutive
solutions are in bijective correspondence with non-degenerate cycle sets, i.e. 
\[
\{\text{non-degenerate involutive solutions}\}
\longleftrightarrow
\{\text{non-degenerate cycle sets}\}.
\]
The correspondence is given by the following formulas: If $(X,r)$ is a
solution, then $(X,\cdot)$, where $x\cdot y=\tau_x^{-1}(y)$, is a
non-degenerate cycle set.  Conversely, if $(X,\cdot)$ is a cycle set, then
$(X,r)$, where 
\[
r(x,y)=\left( (y*x)\cdot y, y*x \right) 
\]
is a non-degenerate involutive solution, where $y*x=z$ if and only if $y\cdot z=x$. 

\begin{exa}
\label{exa:size3}    
    Let $X=\{1,2,3\}$ and $r(x,y)=(\sigma_x(y),\tau_y(x))$, where
    \begin{align*}
        \sigma_1=\id, && \sigma_2=\sigma_3=(23),
        && \tau_1=\tau_2=\tau_3=(23).
    \end{align*}
    This solution is involutive, decomposable and multipermutation. 
    The map $T\colon X\to X$ is the permutation $(23)$. The cycle set associated
    to $(X,r)$ is given by the permutations
    \[
    \varphi_1=\id,\quad
    \varphi_2=(23),\quad
    \varphi_3=(23),\quad
    \]
\end{exa}

\begin{exa}
    \label{exa:size4}
    Let $X=\{1,2,3,4\}$ and $r(x,y)=(\sigma_x(y),\tau_y(x))$, where
    \begin{align*}
        \sigma_1=(34), && \sigma_2=(1324),&&\sigma_3=(1423),&&\sigma_4=(12),\\
        \tau_1=(24),&&\tau_2=(1432),&&\tau_3=(1234),&&\tau_4=(13).
    \end{align*}
    This involutive solution is irretractable, so it is not multipermutation. The permutation group of $(X,r)$ is
    $\mathcal{G}(X,r)=\langle (34),(1324),(1423),(12)\rangle\simeq\D_4$, the dihedral group of 
    order eight. This group acts transitively on $X$, so $(X,r)$ is indecomposable. Note that
    $T\colon X\to X$ is the permutation $(23)$.  The cycle set associated
    to $(X,r)$ is given by the permutations    
    \[
    \varphi_1=(24),\quad
    \varphi_2=(1234),\quad
    \varphi_3=(1432),\quad
    \varphi_4=(13).
    \]
\end{exa}

The solutions $(X,r)$ and $(Y,s)$ are isomorphic if and only if their associated cycle sets are isomorphic, which means that there is a
bijective map $f\colon X\to Y$ such that 
$f(x_1\cdot x_2)=f(x_1)\cdot f(x_2)$
for all $x_1,x_2\in X$. Note that one can write this 
formula as 
\[
f\varphi_{x}f^{-1}=\varphi_{f(x)}
\]
for all $x\in X$. 

One can translate the definitions given at the beginning of the section 
in the language of cycle sets. For example, 
the permutation group of a cycle set $(X,\cdot)$ is then defined as the group
generated by the set $\{\varphi_x:x\in X\}$, and a cycle set
is said to be indecomposable (resp. decomposable) if its permutation group acts
transitively (resp. intransitively) on $X$.  

For a cycle set $(X,\cdot)$ let $T\colon X\to X$ be the map given by 
$T(x)=x\cdot x$. By definition, the cycle set is non-degenerate if and only if the map $T$
is bijective.  In~\cite[Proposition 2.2]{MR1722951}, Etingof, Schedler and
Soloviev proved that $T$ is always bijective whenever the solution is finite,
thus finite cycle sets are regular.  This was proved independently by
Rump in~\cite{MR2132760}.

A cycle set $(X,\cdot)$, where $X=\{1,2,\dots,n\}$, 
is encoded in a table
\[
	M=(M_{i,j})_{1\leq i,j\leq n},\quad
	M_{i,j}=\varphi_{i}(j)=i\cdot j.
\]
This means that the rows of $M$ are the permutations
$\varphi_1,\dots,\varphi_n$ defining the cycle set structure on $X$. The
\emph{principal diagonal} of $M$ is precisely the bijective map $T\colon X\to X$, $x\mapsto
x\cdot x$. 

\begin{exa}
    The cycle set
    corresponding to the solution of Example \ref{exa:size3} 
    can be described by the matrix
    \[
    M=\begin{pmatrix}
    1&2&3\\
    1&3&2\\
    1&3&2
    \end{pmatrix}.
    \]
\end{exa}

\begin{exa}
    The cycle set
    corresponding to the solution of Example \ref{exa:size4} 
    can be described by the matrix
    \[
    M=\begin{pmatrix}
    1&4&3&2\\
    2&3&4&1\\
    4&1&2&3\\
    3&2&1&4
    \end{pmatrix}.
    \]
\end{exa}

To construct all involutive solutions we need to find all possible matrices $M\in\Z^{n\times n}$ with
coefficients in $\{1,2,\dots,n\}$ such that 
\begin{enumerate}
	\item for each $i$ the elements $M_{i,j}$ are all different, 
	\item the elements of the principal diagonal of $M$ are all
		different, and 
	\item $M_{M_{i,j},M_{i,k}}=M_{M_{j,i},M_{j,k}}$ holds for all $i,j,k\in\{1,\dots,n\}$. 
\end{enumerate}

Since the map $T$ is bijective, the diagonal $(M_{i,i})_{1\leq i\leq n}$ has $n$ different elements. This fact is used to reduce our search space. The general idea goes back to Plemmons~\cite{MR0258994}, but in our particular case is based on the following lemma:

\begin{lem}
	\label{lem:trick}
	Let $n\in\N$ and $(X,\cdot)$ be a cycle set of size $n$. Let $T\colon X\to
	X$, $T(x)=x\cdot x$ and $T_1\in\Sym_n$. If $T_1$ and $T$ are conjugate, then there
	exists a cycle set structure $\bullet$ on $X$ such that 
	$(X,\bullet)\simeq (X,\cdot)$ and 
	$T_1(x)=x\bullet x$
	for all $x\in X$.
\end{lem}

\begin{proof}
	Let $\gamma\in\Sym_n$ be such that $T_1=\gamma^{-1}T\gamma$. A direct
	calculation shows that the operation $y\bullet z=\gamma^{-1}(\gamma(y)\cdot
	\gamma(z))$ turns $X$ into a cycle set isomorphic to $(X,\cdot)$ and such
	that 
	\[
		y\bullet y=\gamma^{-1}(\gamma(y)\cdot\gamma(y))=\gamma^{-1}(T(\gamma(y)))=(\gamma^{-1}T\gamma)(y)
	\]
	holds for all $y\in X$. 
\end{proof}

Lemma~\ref{lem:trick} implies that there are only a small number of diagonals
to consider, each diagonal being a representative of a conjugacy class in the
symmetric group $\Sym_n$.  Thus the original problem is divided into $p(n)$
problems, where $p(n)$ is the number of partitions of $n$. In the particular
case of solutions of size nine, this means that there are $p(9)=30$ independent
cases to consider. For size ten, there are
$p(10)=42$ independent cases to consider. 

To construct non-isomorphic solutions 
we shall need the following notation: If $g\in\Sym_n$ and $M$ is a matrix, we
denote by $M^g$ the matrix given by
\[
	(M^g)_{i,j}=g^{-1}\left(M_{g(i),g(j)}\right)
\]
for all $i,j\in\{1,\dots,n\}$.
To avoid expensive isomorphism checking, we are interested in those matrices
$M$ such that
\begin{equation}
	\label{eq:lex}
	M\leq_{\operatorname{lex}} M^g
\end{equation}
for all $g$ in the centralizer $C_{\Sym_n}(T) $ of the permutation $T$ in $\Sym_n$, where
$\operatorname{lex}$ stands for the lexicographic ordering given by
$A\leq_{\operatorname{lex}}B$ if and only if  
\begin{multline*}
(A_{1,1},A_{1,2},\dots,A_{1,n},A_{2,1},A_{2,2},\dots,A_{n,n})\\
\leq
(B_{1,1},B_{1,2},\dots,B_{1,n},B_{2,1},B_{2,2},\dots,B_{n,n})
\end{multline*}
with lexicographical order.  This symmetry breaking is in general very hard to
implement, as in this case the number of constraints will be extremely large.
This happens for example when $T=\id$ or $T$ is a transposition. To deal with this problem, we consider the
constraint~\eqref{eq:lex} only for those permutations that belong to a certain
subset $S$ of $\Sym_n$. This is called \textit{partial symmetry breaking}, see Remark~\ref{re:Schoice} below for details.
 It should be noted that the use of proper subsets of the centralizer of $T$  produce some superfluous solutions and hence some repetitions should be removed by other
computational methods.

We briefly describe our method to remove repetitions. 
Constraint satisfaction methods produce a list of solutions. 
Among the solutions in this list, only those minimal in their orbits are needed. 
A \GAP~script 
parses the list 
and, for each
solution, checks whether or not the solution is minimal with respect to the lexicographic ordering
inside its orbit. This stage of the process is
mostly related with permutation groups and, in general,   
\GAP~deals with them 
in a friendly and very efficient way. 

The enumeration of involutive solutions of size $\leq8$ first appeared
in~\cite{MR1722951}. Table~\ref{tab:size9} shows some numbers
corresponding to solutions of size $\leq10$. New results are presented in shaded
cells. It should be noted that our numbers differ sightly from those
of~\cite[Table 1]{MR1722951}, as our table contains two solutions of size eight that are not
present in previous calculations. 

Our approach with constraint programming needs about ten minutes to construct all
those solutions of size $\leq8$ up to isomorphism.  The calculations
for solutions of size nine took less than four hours and for size ten it took several days, 
see Tables~\ref{tab:time},~\ref{tab:time_gens} and \ref{tab:time10} for some runtimes. They were both performed in an
Intel(R) Core(TM) i7-8700 CPU @ 3.20GHz, with 32GB RAM. The database of involutive solutions of size $\leq9$ needs
about 90MB. Almost 2GB are needed to store all involutive solutions of size $10$.  

\begin{table}[H]
\begin{tabular}{|r|ccccccccc|}
\hline
$n$ & 2 & 3 & 4 & 5 & 6 & 7 & 8 & 9 & 10\tabularnewline
\hline 
solutions & 2 & 5 & 23 & 88 & 595 & 3,456 & \cellcolor{gray!30}{34,530} & \cellcolor{gray!30}{321,931} &   \cellcolor{gray!30}{4,895,272}\tabularnewline
square-free & 1 & 2 & 5 & 17 & 68 & 336 & 2,041 & \cellcolor{gray!30}{15,534} & \cellcolor{gray!30}{150,957}\tabularnewline
ind. & 1 & 1 & 5 & 1 & 10 & 1 & \cellcolor{gray!30}{100} & \cellcolor{gray!30}{16} & \cellcolor{gray!30}{36}\tabularnewline
m.p. & 2 & 5 & 21 & 84 & 554 & 3,295 & \cellcolor{gray!30}{32,155} & \cellcolor{gray!30}{305,916} & \cellcolor{gray!30}{4,606,440}\tabularnewline
irretractable & 0 & 0 & 2 & 4 & 9 & 13 & 191 & \cellcolor{gray!30}{685} & \cellcolor{gray!30}{3,590}\tabularnewline
\hline
\end{tabular}
\caption{Involutive solutions of size $\leq10$.}
\label{tab:size9}
\end{table}

For size $\leq7$ our calculations coincide with those in
~\cite{MR1722951}, 
but differ by two for $n=8$ when the map $T$ an $8$-cycle (see Examples~\ref{exa:new8a} and \ref{exa:new8b} below).
We contacted the authors of ~\cite{MR1722951} regarding the aforementioned 
discrepancy and they found the missing solutions after
a re-run of their own code.

\begin{exa}
	\label{exa:new8a}
	Let $X=\{1,2,\dots,8\}$ and $r(x,y)=(\sigma_x(y),\tau_y(x))$, where 
	\begin{align*}
		&\sigma_1=\sigma_5=(16345278), && \sigma_2=\sigma_6=(12745638),\\ 
		&\sigma_3=\sigma_7=(12385674), && \sigma_4=\sigma_8=(16785234),\\
    	&\tau_1=\tau_5=(18365472), && \tau_2=\tau_6=(14765832), \\
		&\tau_3=\tau_7=(14325876), && \tau_4=\tau_8=(18725436). 
  \end{align*}
  Then $(X,r)$ is an indecomposable and multipermutation involutive solution. 
\end{exa}

\begin{exa}
	\label{exa:new8b}
	Let $X=\{1,2,\dots,8\}$ and $r(x,y)=(\sigma_x(y),\tau_y(x))$, where 
	\begin{align*}
	 &\sigma_1=\sigma_5=(1278)(3456), && \sigma_2=\sigma_6=(1238)(4567),\\
	 &\sigma_3=\sigma_7=(1234)(5678), && \sigma_4=\sigma_8=(1678)(2345),\\
     &\tau_1=\tau_5=(1832)(4765), && \tau_2=\tau_6=(1432)(5876),\\
	 &\tau_3=\tau_7=(1876)(2543), && \tau_4=\tau_8=(1872)(3654).
	\end{align*}
  Then $(X,r)$ is an indecomposable and multipermutation involutive solution. 
\end{exa}

\begin{rem}
The involutive solutions of Examples~\ref{exa:new8a} and~\ref{exa:new8b} are multipermutation and indecomposable solutions. This means that there are 34,530 solutions of size eight, 
100 of them are indecomposable   
and 39 are multipermutation and indecomposable. 
\end{rem}

\begin{rem}\label{re:Schoice}
As mentioned before, for some diagonals $T$ the centralizer
turns out to be too big for 
our computational resources.
A sample $S$ of elements of $C_{\Sym_n}(T)$ is to be chosen
to make the constraint computations feasible.
To construct solutions of size
$n\in\{9,10\}$,
taking $S$ as the full centralizer $C_{\Sym_n}(T)$ of the permutation $T$ in $\Sym_n$ works well for small
centralizers.
For big centralizers, as it is the case when $T=\id$
or a transposition, the standard heuristic local search suggests to look at
the subset of $C_{\Sym_n}(T)$ consisting of permutations moving 
a small number of points of $\{1,2,\dots,n\}$ (at most three usually suffices), 
as most violations of the minimality condition involve few entries of the matrix. 
We also include a small generating set of $C_{\Sym_n}(T)$, since we do not want to lose information by 
inadvertently ignoring permutations that change 
certain labels.
These particular choices of sets $S$ work  well in our setting and allow us to construct solutions in a reasonable time.
{The partial symmetry breaking technique described in this paragraph 
and some of its variations were studied by 
McDonald and Smith in 
\cite{mcdonald2002partial} and the automation of these techniques were studied by Jefferson and Petrie in \cite{10.1007/978-3-642-23786-7_55}}.
\end{rem}

\begin{table}[H]
	\begin{tabular}{|c|c|c|c|}
		\hline 
		$n$ & $T$ & Solutions & CPU time  \tabularnewline
		\hline 
 		9 & (123456789)& 9 & 3 minutes\tabularnewline
 		& (12345678)& 104 & 6 minutes \tabularnewline 
		& (1234567) & 35  & 2 minutes \tabularnewline
		& (123456) & 1,176 & 2 minutes  \tabularnewline
		
        \hline 
		10 & (123456789a) & 20 & 10 hours\tabularnewline
        & (123456789) & 81 & 11 hours\tabularnewline
        & (12345678) & 720 & 9 hours\tabularnewline
        & (1234567) & 238  & 2 hours\tabularnewline
        & (123456) & 9,103 & 2 hours\tabularnewline
        \hline
\end{tabular}
\caption{Some runtimes for constructing involutive solutions of size $n\in\{9,10\}$ with  $S=C_{\Sym_n}(T)$. In these cases there is no need to check if some solutions are isomorphic.}
\label{tab:time}
\end{table}

\begin{table}[H]
	\begin{tabular}{|c|c|c|c|}
		\hline 
		$n$ & $T$ & Solutions & CPU time  \tabularnewline
		\hline 
 		9 & (12345)& 780 & 2 minutes\tabularnewline
 		& (1234) & 11,320 & 3 minutes\tabularnewline 
 		& (123)& 13,061 & 4 minutes \tabularnewline
 		& (12)(34)(56)(78) & 24,345 & 6 minutes
 		\tabularnewline
 		& (12)(34)(56) & 52,866 & 4 minutes\tabularnewline
		& (12)(34) &  61,438 & 8 minutes \tabularnewline
		& (12) & 41,732 & 50 minutes  \tabularnewline
        \hline
\end{tabular}
\caption{Some runtimes for constructing involutive solutions of size nine 
with $S$ being a generating set of $C_{\Sym_n}(T)$.}
\label{tab:time_gens}
\end{table}

\begin{table}[H]
	\begin{tabular}{|c|c|c|c|}
		\hline 
		$n$ & $T$ & Solutions & CPU time  \tabularnewline
		\hline
		9 & (12345) & 780 & 1 minute\tabularnewline
		& (1234) & 11,320 & 1 minute\tabularnewline
	    & (123) & 13,061 & 2 minutes \tabularnewline
		& (12)(34)(56)(78) & 24,345 & 17 minutes\tabularnewline
		& (12)(34)(56) & 52,866 & 9 minutes\tabularnewline
		& (12)(34) & 61,438 & 7 minutes  \tabularnewline
		& (12) & 41,732 & 11 minutes \tabularnewline
		& $\id$ & 15,534 & 1 hour 20 minutes \tabularnewline
		\hline 
		10 & (123) & 143,267 & 2 days \tabularnewline
		 & (12)(34)(56)(78)(9a) & 178,782 & 2 days 7 hours\tabularnewline 
		& (12)(34)(56)(78) & 560,592 & 2 days \tabularnewline
		& (12)(34)(56) & 855,536 & 10 hours\tabularnewline
		& (12)(34) & 807,084  & 8 hours  \tabularnewline
		& (12) & 474,153 & 17 hours \tabularnewline
		& $\id$ & 150,957 & 6 days\tabularnewline
		\hline 
\end{tabular}
\caption{Some runtimes for constructing involutive solutions of size $n\in\{9,10\}$. In these cases $S$ is the set of permutations of $C_{\Sym_n}(T)$ that move $\leq3$ points.}
\label{tab:time10}
\end{table}

In~\cite{MR2095675} Gateva--Ivanova conjectured that all finite square-free solutions
are retractable.  Despite the fact that the conjecture holds in several cases
(see \cite{MR4062375,MR2652212,MR2885602,MR3935814}) a counterexample was found in
\cite{MR3437282}. From a given counterexample one then constructs other counterexamples by
different methods, see~\cite{MR3719300,MR4107577}. It turns out that 
constructing essentially 
new counterexamples to the conjecture seems to be quite challenging. 

For $n\in\N$ let $g(n)$ be the number of isomorphism classes of counterexamples
to Gateva--Ivanova conjecture. Computer calculations show that $g(n)=0$ if $n\leq 7$. Other values of $g(n)$ are shown in Table~\ref{tab:g(n)}. 

\begin{table}[H]
\begin{tabular}{|c|cccc|}
\hline
$n$ & 8 & 9 & 10 & 11\tabularnewline
\hline
$g(n)$ & 1 & 5 & 12 & 23\tabularnewline
\hline
\end{tabular}

\caption{Some values of $g(n)$.}
\label{tab:g(n)}
\end{table}

The determination of the exact value of $g(9)$ took about 7 minutes, $g(10)$ took
3 hours  and $g(11)$ took four days. 
The calculations were performed in an 
Intel(R) Xeon(R) CPU
E5-2670, 2.60GHz, with 32GB RAM. 

\section{Non-involutive solutions}
\label{YB}

The method presented in Section~\ref{IYB} is now used to compute non-involutive
solutions. This time, we translate the problem into the language of skew cycle
sets. First we need basic definitions of the theory of racks.  

\subsection{Racks}
\label{racks}
A \emph{rack} is a pair $(X,\triangleright)$, where $X$ is a set and $X\times X\to X$, $(x,y)\mapsto x\triangleright y$, is a binary operation on $X$ such that the
following conditions are satisfied:
\begin{enumerate}
	\item Each map $X\to X$, $y\mapsto x\triangleright y$ is
		bijective, and
	\item $x\triangleright (y\triangleright z)=(x\triangleright
		y)\triangleright (x\triangleright z)$ for all $x,y,z\in X$.
\end{enumerate}

We can use the ideas presented in the previous section to construct finite racks up to isomorphisms.
However, algorithms to construct and enumerate finite racks of small size are
already known, see for example in~\cite{MR3665565,
MR3118951,MR3904151,MR3957904}. 

As we need racks to construct arbitrary solutions to the YBE, it is convenient 
to recall that the construction problem for racks can be formulated as follows: We need to find
all matrices $R\in\Z^{n\times n}$ with coefficients in $\{1,2,\dots,n\}$ such that 
\begin{enumerate}
	\item for each $i$ the elements $R_{i,j}$ are all different, 
	\item the elements of the principal diagonal of $R$ are all
		different, and 
	\item $R_{i,R_{j,k}}=R_{R_{i,j},R_{i,k}}$ holds for all $i,j,k\in\{1,\dots,n\}$.
\end{enumerate}

To construct racks we can use the trick of considering only representatives of
conjugacy classes of the diagonal and then keep only those matrices which are
minimal in their orbits, with respect to the lexicographical order. 

For $n\in\N$, let $r(n)$ be the number of isomorphism classes of racks of size
$n$. Some values of $r(n)$ appear in Table~\ref{tab:racks}. These values of $r(n)$ were computed by our method based on constraint programming. A better approach to the enumeration of racks of small size appears in ~\cite{MR3957904}.  

\begin{table}[H]
\begin{tabular}{|c|cccccccc|}
\hline
$n$ & 2 & 3 & 4 & 5 & 6 & 7 & 8 & 9\tabularnewline
\hline
$r(n)$ & 2 & 6 & 19 &74&353 & 2,080 & 160,23 & 159,526\tabularnewline
\hline
\end{tabular}
\caption{Enumeration of racks.}
\label{tab:racks}
\end{table}

\subsection{Non-involutive solutions}

The theory of cycle sets can be generalized to deal with
non-involutive solutions to the YBE, see for example~\cite{MR3881192}. 
A \emph{skew cycle set} is a triple
$(X,\cdot,\triangleright)$, where $(X,\triangleright)$ is a rack and $X\times
X\to X$, $(x,y)\mapsto x\cdot y$, is a binary operation such that
\begin{enumerate}
	\item The maps $\varphi_x\colon X\to X$, $y\mapsto x\cdot y$, are
		bijective, 
	\item $\displaystyle{(x\cdot (x\triangleright y))\cdot (x\cdot z)=(y\cdot
		x)\cdot (y\cdot z)}$ for all $x,y,z\in X$, and
	\item $\displaystyle{x\cdot (y\triangleright z)=(x\cdot y)\triangleright
			(x\cdot z)}$ for all $x,y,z\in X$. 
\end{enumerate}
As it happens in the involutive case, finite solutions to the YBE are in
bijective correspondence with skew cycle sets, i.e. 
\begin{equation}
	\label{correspondence}
\{\text{non-degenerate solutions}\}
\longleftrightarrow
\{\text{non-degenerate skew cycle sets}\}
\end{equation}
The correspondence is given as follows. If $(X,r)$ is a solution, 
then the skew cycle set on $X$ is given by
\begin{align*}
	&x\cdot y=\tau_x^{-1}(y),
	&&x\triangleright y=\tau_x\sigma_{\tau_y^{-1}(x)}(y).
\end{align*}
Conversely, if $X$ is a skew cycle set, then 
\[
	r(x,y)=\left( (y*x)\cdot ( (y*x)\triangleright y),y*x \right)
\]
is a solution, where $y*x=z$ if and only if $y\cdot z=x$.  We refer
to~\cite{MR3974961} for more information on the interaction between solutions
and their associated racks. 

\begin{rem}
	Under the bijective correspondence~\eqref{correspondence}, involutive
	solutions correspond to the cycle sets defined in Section \ref{IYB}. 
\end{rem}

We now translate the problem of constructing all
finite solutions into a problem suitable for constraint
programming. Given a matrix $R$ corresponding to a rack of size $n$, we want to
find all possible matrices $M\in\Z^{n\times n}$ with coefficients in
$\{1,2,\dots,n\}$ such that 
\begin{enumerate}
	\item for each $i$ the elements $M_{i,j}$ are all different, 
	\item the elements of the principal diagonal of $M$ are all
		different, 
	\item $M_{M_{i,R_{i,j}},M_{i,k}}=M_{M_{j,i},M_{k,l}}$ holds for all $i,j,k\in\{1,\dots,n\}$, and
	\item $M_{i,R_{j,k}}=R_{M_{i,j},M_{i,k}}$ for all $i,j,k\in\{1,\dots,n\}$. 
\end{enumerate}

We can exclude the trivial rack from our algorithm, as involutive solutions
were computed in Section~\ref{IYB}. It only remains to deal with the isomorphism problem. 
Thus we are interested in those matrices $M$ such that
\[
	M\leq_{\operatorname{lex}} M^g
\]
for all $g$ in the stabilizer of the rack $R$, where $\operatorname{lex}$
stands for the lexicographic ordering on matrices described in
Section~\ref{IYB}.  This symmetry breaking is in general easy to implement, as
stabilizers of racks tend to be small.  

For $n\in\N$ let $s(n)$ be the number of isomorphism classes of non-involutive solutions of size $n$. We summarize our calculations in
Table~\ref{tab:non_involutive}.

\begin{table}[H]
\begin{tabular}{|c|ccccccc|}
\hline
$n$ & 2 & 3 & 4 & 5 & 6 & 7 & 8\tabularnewline
\hline
$s(n)$ & 2 & 21 & 230 & 3,519 & 100,071 & 4,602,720 & 422,449,480\tabularnewline
\hline
\end{tabular}
\caption{Enumeration of non-involutive solutions.}
\label{tab:non_involutive}
\end{table}

The calculations for $s(n)$ for all $n\leq 6$ took about 10 minutes, $s(7)$ needed 2 hours 
and 17 minutes and $s(8)$ took about one day. The database of non-involutive solutions 
needs about 750MB for solutions of size $\leq7$ and 
around 100GB for solutions of size eight. 

\subsection{Non-involutive biquandles} 

Recall that a \emph{biquandle} is a solution such that
its associated rack is a quandle, that means that  \[
x\triangleright x=\tau_x\sigma_{\tau_x^{-1}(x)}(x)=x
\]
for all $x\in X$. In particular, 
involutive solutions are biquandles. Enumeration of biquandles of small size  appear in~\cite{MR2819176,MR3666513,MR2246021}. 

For $n\in\N$ let $b(n)$ be the number of isomorphism classes of non-involutive biquandles of size $n$. 
The enumeration of non-involutive biquandles
appear in Table~\ref{tab:biquandles}.

\begin{table}[H]
\begin{tabular}{|c|cccccc|}
\hline
$n$ & 3 & 4 & 5 & 6 & 7 & 8\tabularnewline
\hline
$b(n)$ & 10 & 75 & 974 & 18,548 & 621,414 & 37,836,551\tabularnewline
\hline
\end{tabular}
\caption{Enumeration of non-involutive biquandles.}
\label{tab:biquandles}
\end{table}










\subsection*{Acknowledgments} 
We thank A. Bartholomew, M. Farinati, P. Etingof and T. Schedler for useful conversations. 
The second author is partially supported by PICT 2018-3511
and is also a Junior Associate of the ICTP. The third  author 
acknowledges support of 
NYU-ECNU Institute of Mathematical Sciences at NYU--Shanghai and he 
is supported in part by PICT 2016-2481 and UBACyT 20020170100256BA.


%

\bibliographystyle{abbrv}
\bibliography{refs}

\end{document}